\documentclass[11pt]{amsart}
\usepackage{amssymb,amsmath}
\usepackage{verbatim}

\numberwithin{equation}{section}

\newtheorem{thm}[equation]{Theorem}
\newtheorem{conjecture}[equation]{Conjecture}
\newtheorem{cor}[equation]{Corollary}
\newtheorem{prop}[equation]{Proposition}

\newtheorem{defi}[equation]{Definition}

\theoremstyle{remark}
\newtheorem{remark}[equation]{Remark}

\newtheorem*{acknowledgments}{Acknowledgments}

\renewcommand{\bar}[1]{#1\llap{$\overline{\phantom{\rm#1}}$}}
\newcommand{\lra}{\longrightarrow}
\DeclareMathOperator{\tot}{{all}}

\DeclareMathOperator{\sep}{{sep}}
\DeclareMathOperator{\tor}{{tor}}
\DeclareMathOperator{\End}{{End}}
\DeclareMathOperator{\Gal}{{Gal}}
\DeclareMathOperator{\trdeg}{{trdeg}}

\newcommand{\N}{{\mathbb N}}
\newcommand{\A}{{\mathbf{A}}}
\newcommand{\Gammabar}{{\overline{\Gamma}}}
\newcommand{\F}{{\mathbb{F}}}
\newcommand{\Fbar}{{\overline{\F}}}

\newcommand{\Kbar}{{\overline{K}}}

\newcommand{\bG}{{\mathbb G}}
\newcommand{\OO}{{\mathcal O}}

\newcommand{\cV}{\mathcal{V}}

\newcommand{\isomto}{\overset{\sim}{\rightarrow}}
\newcommand{\tensor}{\otimes}

\newcommand{\Ga}{{{\mathbb G}_a}}

\begin{document}

%#######################################################################
%#######################################################################

% Declarations for Front Matter

\title[Ad\`{e}lic closure of a Drinfeld module]{Algebraic equations on the ad\`{e}lic closure of a Drinfeld module}

\author{Dragos Ghioca}

\address{
Dragos Ghioca \\
Department of Mathematics \\
University of British Columbia \\
1984 Mathematics Road\\
Vancouver, BC V6T 1Z2\\
Canada
}

\email{dghioca@math.ubc.ca}

\author{Thomas Scanlon}

\address{
Thomas Scanlon\\
Department of Mathematics\\
Evans Hall\\
University of California\\
Berkeley, CA 94720-3840\\
USA
}

\email{scanlon@math.berkeley.edu}

%#######################################################################
%#######################################################################

\begin{abstract}
Let $k$ be a field of positive characteristic and $K = k(V)$ a function field of a variety $V$ over $k$ and let $\A_K$ be a ring of ad\'{e}les of $K$ with respect to a cofinite set of the places on $K$ corresponding to the divisors on $V$.  Given a Drinfeld module $\Phi:{\mathbb F}[t] \to \operatorname{End}_K({\mathbb G}_a)$ over $K$ and a positive integer $g$ we regard both $K^g$ and $\A_K^g$ as $\Phi({\mathbb F}_p[t])$-modules under the diagonal action induced by $\Phi$.  For $\Gamma \subseteq K^g$ a finitely generated $\Phi(\F_p[t])$-submodule and an affine subvariety $X \subseteq \bG_a^g$ defined over $K$, we study the intersection of $X(\A_K)$, the ad\`{e}lic points of $X$, with $\overline{\Gamma}$, the closure of $\Gamma$ with respect to the ad\`{e}lic topology, showing under various hypotheses that this intersection is no more than $X(K) \cap \Gamma$.
\end{abstract}

\date{\today}

\maketitle

%#######################################################################
%#######################################################################
%#######################################################################

\section{Introduction}

Fix a field $K$ with a set of places $\Omega$ so that each $x \in K$ is $v$-integral at all but finitely many $v \in \Omega$.  We say that the Hasse principle holds for a given class ${\mathcal C}$ of algebraic varieties over $K$ if for each $X \in {\mathcal C}$ the
set $X(K)$ of $K$-rational points on $X$ is nonempty just in case for each $v \in \Omega$ the set $X(K_v)$ of $K_v$-rational points is nonempty where $K_v$ is the completion of $K$
 with respect to $v$.

In some ways, the Hasse principle is both too strong and too weak.  It is too strong in the sense that the classes of varieties for which it is known to hold are quite restrictive,
for example, Brauer-Severi varieties when $K$ is a number field and $\Omega$ is the set of all places, while there are many classes of varieties for which it is known to fail.  It is too
weak in the sense that it merely says that one may test for existence of a $K$-rational point by checking local conditions.  It does not say that the set $X(K)$ may be computed from purely
local data.

Let us suppose now that ${\mathcal C}$ is a class of pairs $(X,Y)$ where each of $Y$ is a variety over $K$ and $X$ is a subvariety of $Y$ also defined over $K$.   Suppose moreover that
for each such pair $(X,Y) \in {\mathcal C}$ that we know $Y(K)$.  We wish to describe $X(K)$ as a subset of $Y(K)$ via strictly local data.  Consider the ring of ad\`{e}les of $K$:
$$\A_K := \left\{ (x_v)_{v \in \Omega} \in \prod_{v \in \Omega} K_v ~\vert~
x_v \in {\mathcal O}_{K_v} \text{ for all but finitely many } v \in \Omega \right\}.$$

From the diagonal inclusion $K \hookrightarrow {\mathbf A}_K$ we obtain inclusions  $X(K) \subseteq Y(K) \subseteq Y({\mathbf A}_K)$ and $X(K) \subseteq X({\mathbf A}_K)$.  Indeed,
we may express $X(K)  = Y(K) \cap X({\mathbf A}_K)$.   The set $Y(K)$ itself is not described by local conditions, but we may think of its closure with respect to the ad\`{e}lic
topology as a locally defined set.  Replacing $Y(K)$ by its closure $\overline{Y(K)}$, in general we have only the inclusion $X(K) \subseteq \overline{Y(K)} \cap X({\mathbf A}_K)$.
When equality holds for all pairs $(X,Y) \in {\mathcal C}$ then we would say that the set of $K$-rational points on $X$ may be computed locally.
In~\cite{Poonen-Voloch}, Poonen and Voloch show that when $K$ is a function field field of a curve, then the local conditions determine $X(K)$ whenever $Y$ is an abelian variety and $X$
is a subvariety.   Moreover, in the case that the field of constants of $K$ is finite, they relate the possible inequality of $X(K)$ with $\overline{Y(K)} \cap X({\mathbf A}_K)$
 to the Brauer-Manin obstruction.

With the above problem, we took the full set $Y(K)$ of $K$-rational points on $Y$ as given.  There is no good reason to do so, especially as in practice, even when $Y$
is an abelian variety it might be possible to determine some points on $Y$ and possibly thereby compute the group they generate, but the full computation of $Y(K)$ may be difficult.
In other cases, some proper subset of $Y(K)$ might have more arithmetic significance.  Let us say that a class ${\mathcal C}$ of triples $(X,Y,\Gamma)$ consisting of an algebraic
 variety $Y$ over $K$, a subvariety $X \subseteq Y$ also over $K$ and a set $\Gamma \subseteq Y(K)$ is locally determined if for each $(X,Y,\Gamma) \in {\mathcal C}$ we have
 $X(K) \cap \Gamma =  X({\mathbf A}_K) \cap \overline{\Gamma}$ in $Y({\mathbf A}_K)$.

In this paper we take $K$ to be a function field in positive characteristic and we address the question of when it happens that $(X,Y,\Gamma)$ is locally
 determined for $Y = {\mathbb A}^m_K$ affine space over $K$,
 $X \subseteq Y$ an arbitrary subvariety and $\Gamma \subseteq Y(K)$ a finitely generated submodule with respect to the action of a Drinfeld module.

Let $p$ be a prime number, $R := {\mathbb F}_p[t]$ and ${\mathbb F} := {\mathbb F}_p(t)$.  For the remainder of this introduction we shall assume that $K$ is a field extension of  ${\mathbb F}$.  We denote by $\operatorname{End}_K({\mathbb G}_a)$ the ring of endomorphisms of the additive group scheme over $K$.  Concretely, $\operatorname{End}_K({\mathbb G}_a)$
  may be identified with the ring of additive polynomials over $K$ under addition and composition.  A Drinfeld module of generic characteristic defined over $K$ is a ring homomorphism   $\Phi:R \to \operatorname{End}_K({\mathbb G}_a)$ for which $\Phi_t := \Phi(t)$ has differential $t$, but is not scalar multiplication. That is,
$$\Phi_t(x) = tx + c_1 x^p + \cdots + c_D x^{p^D} \in K[x] \text{ where } c_D \neq 0 \text{ (and }D \geq 1 \text{).}$$

Let $g$ be any positive integer, then we let $\Phi$ act diagonally on ${\mathbb G}_a^g$.  This induces a natural structure of a $\Phi(R)$-module on $K^g$ where each element $a \in R$ acts via $\Phi_a := \Phi(a)$; when we want to emphasize the $\Phi(R)$-module structure of $K^g$ we will denote it by $K_{\Phi}^g$. We say that $\Gamma \subseteq K^g$ is a $\Phi(R)$-module if it is a submodule of $K_{\Phi}^g$ with respect to this action.
  Our main theorems (Theorem 2.1 and Theorem 2.3) assert that under appropriate hypotheses, detailed in Section~\ref{statement results} and whose necessity or removability are discussed in Remark~\ref{hypotheses removal}, for a subvariety $X \subseteq {\mathbb G}_a^g$ of a power of the additive group and a
finitely generated $\Phi(R)$-submodule $\Gamma \subseteq K^g$, the local conditions determine the intersection of $\Gamma$ with $X$.  That is, $X(K) \cap \Gamma = X({\mathbf A}_K) \cap \overline{\Gamma}$.

We had embarked on this project hoping to thereby prove Denis's Mordell-Lang conjecture for Drinfeld modules~\cite{Denis}.

\begin{conjecture}
\label{DML}
Let $\Phi$ be a Drinfeld module of generic characteristic.
If $X \subseteq {\mathbb G}_a^g$ is any subvariety of a Cartesian power of the additive group and $\Gamma \subseteq K^g$ is a finitely generated $\Phi(R)$-module, then
$X(K) \cap \Gamma$ is a finite union of translates of $\Phi(R)$-submodules of $\Gamma$.
\end{conjecture}

Instead, we have strengthened the arguments employed in the proofs of the known partial results towards Conjecture~\ref{DML} to reduce the problem of computing $X({\mathbf A}_K) \cap \overline{\Gamma}$ to the case that $X$ is
zero dimensional.  This is very much in the spirit of the proofs of the analogous results about semiabelian varieties~\cite{Poonen-Voloch,Sun}.

While the overall structure of our proofs is similar to those for semiabelian varieties, there are some essential differences; the most important of which concerns the problem
of showing that the topology induced on $\Gamma$ from the ad\`{e}lic topology is at least as strong as the topology given by the subgroups $\Phi_a(\Gamma)$ as $a$ ranges
through the nonzero elements of $R$.  In the case of semiabelian varieties, the corresponding result is almost immediate, but here, because the ambient geometry of the additive group
does not determine the Drinfeld module structure, our proof is substantially more complicated.

This paper is organized as follows.  In Section~\ref{statement results} we establish our notation and state our main results.  In Section~\ref{dimension zero} we prove a simple case of our theorem in which
the ad\`{e}lic topology on $\Gamma$ is discrete.  In Section~\ref{first part} we reduce the overall problem to that of zero dimensional $X$. We complete the proof in Section~\ref{second part} by proving
a uniform version of a version of the Drinfeld module Mordell-Lang conjecture.

\medskip

\begin{acknowledgments}
The first author is grateful to Thomas Tucker for many stimulating conversations and for his constant encouragement for writing this paper.  The second author was partially supported by NSF grants FRG DMS-0854998 and DMS-1001556 during the writing of this paper.
\end{acknowledgments}

%#######################################################################
%#######################################################################

\section{Statement of our results}
\label{statement results}

Let $K$ be a finitely generated extension of $\F$ such that $\trdeg_{\F}K\ge 1$. We study the intersection of the ad\`{e}lic points of $V$ with the ad\`{e}lic closure of $\Gamma$ for both cases: generic characteristic and special characteristic Drinfeld modules $\Phi:R\lra \End_K(\bG_a)$. First we need to set up our notaton.

We let $\Kbar$ be a fixed algebraic closure of $K$, and we denote by $K^{\sep}$ the separable closure of $K$ inside $\Kbar$. We also let $\Fbar$ be the algebraic closure of $\F$ inside $\Kbar$.  At the expense of replacing $K$ by a finite extension and then replacing $\F$ by a finite extension, we may assume (see \cite[Remark 4.2]{de Jong}) there exists a projective, smooth, geometrically irreducible variety $\cV$ defined over $\F$ such that its function field equals $K$ (note that $\F$ is the algebraic closure of $\F_p(t)$ inside $K$). We let $\Omega:=\Omega_{\cV}$ be the set of \emph{all} valuations of $K$ associated with the irreducible divisors of $\cV$. By abuse of language, we will also call each $v\in\Omega$ a place of $K$; note that for each nonzero $x\in \F$, we have $|x|_v=1$, where $| \cdot |_v$ is the corresponding norm for the place $v$. Also note that $\F$ embeds naturally into the residue field $\F_v$ of each place $v\in\Omega$; furthermore, each $\F_v$ is a finite extension of $\F$ and the reduction modulo $v$ may be seen as a \emph{specialization} of $K$ to $\Fbar$.
Finally, we note that $K$ admits a product formula with respect to the set of places $\Omega$ (see \cite{Serre}); i.e., there exists a set of positive integers $\{N_v\}_{v\in\Omega}$ such that for all $x\in K^*$ we have
\begin{equation}
\label{product formula}
\prod_{v\in\Omega} |x|_v^{N_v}=1.
\end{equation}

Now, for each $v\in\Omega$, we let $K_v$ be the completion of $K$ at $v$. We let $\OO_v$ be the ring of $v$-adic integers contained in $K_v$, and we equip $K_v$ with the $v$-adic topology. We define the topological ring of ad\`{e}les $\A_K$ as the restricted direct product $\prod_{v\in\Omega_1} (K_v,\OO_v)$ where $\Omega_1$ is a cofinite subset of $\Omega$. More precisely, each element of $\A_K$ is of the form $\{x_v\}_{v\in\Omega_1}$, where $x_v\in K_v$ for all $v$, and moreover, for all but finitely many $v\in\Omega_1$ we actually have that $x_v\in\OO_v$.
In particular, $\prod_{v\in\Omega_1}\OO_v$ is an open subset of $\A_K$ whose induced topology is the usual product topology. We note that one could take instead the ad\`{e}les as the restricted product of $(K_v,\OO_v)$ for \emph{all} places $v\in\Omega$; i.e.,
$$\A_{K,\tot}:=\prod_{v\in\Omega} (K_v,\OO_v).$$
The restriction on $K$ of the topology from $\A_{K,\tot}$ yields a stronger topology than the topology induced from $\A_K$. Therefore, the results we are proving regarding the closures in the ad\`{e}lic topology of subsets of $K$ are stronger with our definition for the ad\`{e}les; see also the next result which is a purely topological statement.

\begin{prop}
\label{essential lemma}
Let $\Gamma\subset K$ be a set of points which are all integral at all places in $\Omega_1$. Let $y=(y_v)_{v\in\Omega}\in \A_{K,\tot}$ be a point in the closure of $\Gamma$ inside $\A_{K,\tot}$. Assume we know that there exists $y_0\in K$ such that $y_v=y_0$ for all $v\in\Omega_1$. Then $y_v=y_0$ for \emph{all} $v\in\Omega$; i.e., $y=y_0$.
\end{prop}

\begin{proof}
Clearly, $y_0$ is integral at each place $v\in\Omega_1$. At the expense of replacing $y$ by $y-y_0$ we may assume $y_0=0$. Therefore, we know that there exists an infinite sequence $\{x_n\}_{n\in\N}\subset K$ satisfying the following conditions:
\begin{enumerate}
\item[(a)] $|x_n|_v\le 1$ for all $v\in\Omega_1$;
\item[(b)] $|x_n|_v\to 0$ as $n\to\infty$,  for all $v\in\Omega_1$; and
\item[(c)] $|x_n-y_v|_v\to 0$ as $n\to\infty$, for all $v\in\Omega\setminus\Omega_1$.
\end{enumerate}
Using condition (c) above, we conclude that there exists $n_0\in\N$ such that for all $n\ge n_0$ we have $|x_n|_v=|y_v|_v$ for all $v\in\Omega\setminus\Omega_1$. In particular, there exists a positive real number $C_0$ such that for all $n\ge n_0$, we have
\begin{equation}
\label{essential one}
\prod_{v\in\Omega\setminus\Omega_1}|x_n|_v^{N_v}\le C_0.
\end{equation}
On the other hand, using conditions (a)-(b) above we conclude that
\begin{equation}
\label{essential two}
\lim_{n\to\infty} \prod_{v\in\Omega_1} |x_n|_v^{N_v} =0.
\end{equation}
Now, if the sequence $\{x_n\}_{n\in\N}$ contains infinitely many nonzero elements, then equations \eqref{essential one} and \eqref{essential two} provide a contradiction to the product formula \eqref{product formula}. Therefore, we obtain that there exists $n_1\in\N$ such that $x_n=0$ for all $n\ge n_1$. Hence for each $v\in\Omega\setminus\Omega_1$ we have $y_v=0$, which concludes our proof.
\end{proof}

Naturally, $\bG_a^g(\A_K)=\A_K^g$, and we define the product topology on $\A_K^g$. Then for any subset $S\subseteq \A_K^g$, we let $\overline{S}$ be the topological closure of $S$ in $\A_K^g$. Clearly, $\A_K$ (and $\A_K^g$) have natural structures of $\Phi(R)$-modules denoted $\A_{K,\Phi}$ (and respectively  $\A_{K,\Phi}^g$).

We will prove the following result.
\begin{thm}
\label{adelic result generic characteristic}
Let $K$ be a function field over $\F$ of arbitrary finite transcendence degree, and let $\Phi$ be a Drinfeld module of generic characteristic defined over $K$. Then for each affine subvariety $V\subseteq \bG_a^g$ defined over $K$, and for each finitely generated $\Phi(R)$-submodule $\Gamma\subseteq K^g$, we have $V(K)\cap\Gamma=V(\A_K)\cap \Gammabar$.
\end{thm}

So, Theorem~\ref{adelic result generic characteristic} shows that the intersection $V(K)\cap\Gamma$ can be recovered from the local information (as given by the ad\`{e}les). Theorem~\ref{adelic result generic characteristic} is an immediate consequence of the fact that for all but finitely many places $v\in\Omega$, the induced $v$-adic topology on $\Gamma$ is the discrete topology (see Proposition~\ref{discrete subgroups}).

Our Theorem~\ref{adelic result generic characteristic} is the Drinfeld modules analogue of \cite[Theorem A]{Poonen-Voloch}. Similarly, we can prove an analogue of \cite[Theorem B]{Poonen-Voloch} for Drinfeld modules of special characteristic.

A Drinfeld module of special characteristic is a ring homomorphism $\Phi: R \lra \End_K(\bG_a)$ such that $\Phi_t=\Phi(t)$ is an inseparable endomorphism of the additive group scheme; more precisely,
$$\Phi_t(x)=\sum_{i=d}^D c_i x^{p^i}\text{ for some $c_i\in K$, where $1\le d\le D$.}$$

In order to state our next result we also need the following definition.
\begin{defi}
\label{full subgroup}
For each $\Phi(R)$-submodule $\Gamma\subseteq K^g$, we call its full divisible hull inside $K$, the set of all $x\in\bG_a^g(K)$ such that for some nonzero $a\in R$ we have $\Phi_a(x)\in \Gamma$.

We say that the subgroup $\Gamma\subseteq K_{\Phi}^g$ is \emph{full in $K$}, if it equals its full divisible hull inside $K_{\Phi}^g$. When the field $K$ is understood from the context, we will simply say that $\Gamma$ is \emph{full}.
\end{defi}

As proved in \cite{Wang} (see also \cite{Poonen} and \cite[Theorem 5.12]{JNT-2} for similar results), the $\Phi(R)$-module structure of $K^g_{\Phi}$ is that of a direct sum of a finite torsion module with a free $\Phi(R)$-module of rank $\aleph_0$. So, replacing $\Gamma$ with the $\Phi(R)$-submodule containing all $x\in K^g$ such that $\Phi_a(x)\in\Gamma$ for some nonzero $a\in R$, we may assume $\Gamma$ is both full in $K$ and finitely generated.

Let $\Phi:R\lra \End_K(\bG_a)$ be any Drinfeld module, and let $L$ be a field containing $K$. For each nonzero $a\in R$ we let $\Phi[a](L)$ be the set of all $x\in L$ such that $\Phi_a(x)=0$. We also let
$$\Phi[t^{\infty}](L):=\bigcup_{n\in\N} \Phi[t^n](L),$$
and
$$\Phi_{\tor}(L):=\bigcup_{a\in R\setminus\{0\}} \Phi[a](L).$$
When $L=\Kbar$, we simply use $\Phi[a]$ and $\Phi_{\tor}$ instead of $\Phi[a](\Kbar)$ and respectively $\Phi_{\tor}(\Kbar)$.
In general, for any $\Phi(R)$-submodule $\Gamma$ we denote by $\Gamma_{\tor}$ the torsion submodule of $\Gamma$.

We also recall the notion of modular transcendence degree, first introduced by the second author in \cite{Sc}.
The \emph{modular transcendence degree} of $\Phi$ is the smallest nonnegative integer $d$ such that there exists a subfield $L\subseteq \Kbar$ of transcendence degree $d$ over $\F_p$, and there exists $\gamma\in \Kbar$ such that
\begin{equation}
\label{conjugation phi}
\gamma^{-1}\Phi_t(\gamma\cdot x)\in L[x].
\end{equation}

\begin{thm}
\label{adelic result special characteristic}
Let $K$ be a function field over $\F$ of arbitrary finite transcendence degree, and let $\Phi$ be a Drinfeld module of special characteristic defined over $K$. Assume $\Phi[t^{\infty}](K^{\sep})$ is finite, and also assume that $\Phi$ has positive modular transcendence degree. Let $V$ be an affine subvariety of $\bG_a^g$ defined over $K$ which contains no translate of a positive dimensional algebraic subgroup of $\bG_a^g$. Then for each finitely generated full $\Phi(R)$-submodule $\Gamma\subseteq K^g$, we have $V(K)\cap\Gamma=V(\A_K)\cap\Gammabar$.
\end{thm}

Since we assumed that $V$ contains no translate of a positive dimensional algebraic subgroup of $\bG_a^g$, the main result of \cite{Sc} (see also our Theorem~\ref{AWZ}) yields that $V(K)\cap\Gamma$ is finite. However, the novelty of our result is that we can show that each of these finitely many points from the intersection $V(K)\cap\Gamma$ can be recovered purely from the local information given by the ad\`{e}lic topology. The strategy for proving Theorem~\ref{adelic result special characteristic} is as follows:
\begin{enumerate}
\item[(a)] in Section~\ref{first part} we show that Theorem~\ref{adelic result special characteristic} holds under the additional hypothesis that $V$ is zero-dimensional; and
\item[(b)] then in Section~\ref{second part} we show that for each subvariety $V$, there exists a zero-dimensional subvariety $V_0$ of $V$ also defined over $K$ such that $V(\A_K)\cap\Gammabar= V_0(\A_K)\cap \Gammabar$. This reduces Theorem~\ref{adelic result special characteristic} to part (a) above.
\end{enumerate}

\begin{remark}
\label{hypotheses removal}
To what extent can we eliminate our hypotheses on the field of definition of $\Phi$, the geometry of $X$ and the fullness of $\Gamma$?

If $\Phi$ is isotrivial, that is, isomorphic to a Drinfeld module defined over a finite field, then the intersection $X(K) \cap \Gamma$ may very well be infinite coming from Frobenius orbits even if $X$ contains no translates of positive dimensional algebraic groups.  However, a properly reformulated Mordell-Lang statement holds in this case~\cite{TAMS} and arguments along the lines of those in~\cite{Sun} should yield the full result in this case.

The case where $\Phi$ is of generic characteristic and modular transcendence degree one, while $K$ is a finite extension of $\mathbb{F}_p(t)$ is intriguing, but as with the number field case of the corresponding problem for abelian varieties (which was not considered in~\cite{Poonen-Voloch}), this appears to be beyond the reach of existing techniques.  Indeed, the Drinfeld module problem may be harder than the number field abelian variety problem as due to Faltings' work~\cite{Faltings} we know the Mordell-Lang conjecture to be true, but Denis's conjecture is still open.

At one level, the restriction that $X$ not contain a translate of a positive dimensional algebraic group is necessary as if $X$ is a group  which meets $\Gamma$ in a Zariski dense set and the ad\`{e}lic topology on $\Gamma$ is not discrete, as will be the case if $\Omega$ consists only of places where $\Phi$ has good reduction and $\Gamma$ is integral, then necessarily $X(\A_K) \cap \overline{\Gamma} \neq X(K) \cap \Gamma$.  However, we conjecture that the assertion should be $X(\A_K) \cap \overline{\Gamma} = \overline{X(K) \cap \Gamma}$.   Our obstruction to proving this conjecture is that we do not have a good description of the connected algebraic groups which meet finitely generated $\Phi(R)$-modules in Zariski dense sets.  Denis's conjecture for Drinfeld modules of generic characteristic would imply that such a group must be a $\Phi(R)$-module, but the first author has produced counter-examples for some non-isotrivial Drinfeld modules of special characteristic~\cite{IMRN}.  Nevertheless, these examples are all $\Phi_t$-periodic and it would be enough to know that any such group is $\Phi_t$-periodic.

Finally, the requirement that $\Gamma$ be full becomes moot in our revised conjecture just as it was not an issue in~\cite{Poonen-Voloch} as the group there was assumed to be the full group of rational points.  While we were unable to produce an example showing the necessity of this hypothesis, we believe that it cannot be removed.
\end{remark}

%#######################################################################
%#######################################################################

\section{Proof of Theorem~\ref{adelic result generic characteristic}}
\label{dimension zero}

We are working with the notation and the hypotheses of Theorem~\ref{adelic result generic characteristic} (see Section~\ref{statement results}).

For each finitely generated $\Phi(R)$-submodule $\Gamma$ of $K^g_{\Phi}$, we have $\Gamma\subseteq \Gamma_0^g$, where $\Gamma_0\subseteq K_{\Phi}$ is the $\Phi(R)$-submodule generated by all projections of $\Gamma$ on each coordinate of $\bG_a^g$.
Clearly, $\Gamma_0$ is also a finitely generated $\Phi(R)$-submodule. We will show that each finitely generated $\Phi(R)$-submodule of $K_{\Phi}$ is discrete with respect to all but finitely many places $v\in\Omega$. In particular this yields that $\Gamma$ is discrete $v$-adically for almost all places $v$. Thus $\Gammabar=\Gamma$, and hence Theorem~\ref{adelic result generic characteristic} follows immediately.

\begin{prop}
\label{discrete subgroups}
Let $\Gamma\subseteq K_{\Phi}$ be a finitely generated $\Phi(R)$-submodule, and let $v$ be a place in $\Omega$ at which all coefficients of $\Phi_a$ (for each $a\in R$) and also all elements of $\Gamma$ are integral. Then $\Gamma$ is discrete as a subset of $K_v$.
\end{prop}
\begin{proof}
We have to show that there is no infinite sequence of points in $\Gamma$ converging to $0$ with respect to the $v$-adic topology.
We prove Proposition~\ref{discrete subgroups} by induction on the number $r$  of generators for the $\Phi(R)$-module $\Gamma$.

Assume $r=1$, i.e. $\Gamma$ is the cyclic $\Phi(R)$-module generated by $x_1$, and also assume $\Gamma$ is not discrete with respect to the $v$-adic topology. In particular, this yields that $x_1$ is not torsion since otherwise $\Gamma$ would be finite (and thus discrete $v$-adically).

Therefore, there exists a sequence of distinct nonzero elements $a_n\in R$ (for each $n\in\N$) such that $|\Phi_{a_n}(x_1)|_v\to 0$ as $n\to \infty$.
Now, since each coefficient of $\Phi_a$, for each $a\in R$ is integral at $v$, we obtain that the set
$$I_1:=\{a\in R\text{ : }|\Phi_a(x_1)|_v<1\}$$
is an ideal (see also \cite[Lemma 3.11]{equidrin}). Since $R$ is a PID, we conclude that $I_1$ is generated by some polynomial $c_1$. Furthermore, since we assumed that $|\Phi_{a_n}(x_1)|_v\to 0$, we obtain that $I_1$ is nontrivial, and thus $c_1\ne 0$.
So, for each $n\in\N$ such that $|\Phi_{a_n}(x_1)|_v<1$ we have that $c_1\mid a_n$, and thus there exists a nonzero $b_n\in R$ such that $a_n=b_n\cdot c_1$. On the other hand, $|b_n|_v=1$ since each element of $\F_p(t)^*\subseteq \F^*$ is a $v$-adic unit. According to \cite[Lemma 3.10]{equidrin}, we obtain then that $$|\Phi_{a_n}(x_1)|_v=|\Phi_{b_n}(\Phi_{c_1}(x_1))|_v=|b_n\cdot \Phi_{c_1}(x_1)|_v=|\Phi_{c_1}(x_1)|_v,$$
which contradicts our assumption that $\Phi_{a_n}(x_1)$ converges $v$-adically to $0$.

Now, assume we proved that Proposition~\ref{discrete subgroups} holds for each finitely generated $\Phi(R)$-module spanned by less than $r$ elements; next we prove that it holds when $\Gamma$ is generated by $r$ elements. Since $\Gamma$ is a finitely generated $\Phi(R)$-submodule, it is a finite union of cosets of free finitely generated $\Phi(R)$-submodules. Therefore, it suffices to prove Proposition~\ref{discrete subgroups} for free submodules.
Let $x_1,\dots,x_r$ be a basis for the free $\Phi(R)$-module $\Gamma$.

Assume there exists an infinite sequence $\{(a_{n,1},\cdots,a_{n,r})\}_{n\in\N}$ of nonzero tuples of elements of $R$ (i.e. for each $n\in\N$ not all $a_{n,i}=0$) such that
$$\left|\sum_{i=1}^r \Phi_{a_{n,i}}(x_i)\right|_v\to 0\text{ as }n\to \infty.$$
The set $I_r$ of tuples $(c_1,\dots,c_r)\in R^r$ such that $|\sum_{i=1}^r \Phi_{c_i}(x_i)|_v<1$ is an $R$-submodule of $R^r$ (again using \cite[Lemma 3.11]{equidrin} since all coefficients of each $\Phi_a$ are $v$-adic integers). Since $R$ is a PID and $R^r$ is a free $R$-module, then $I_r$ is a free $R$-module of rank $s\le r$. Let $\{(c_{j,1},\cdots,c_{j,r})\}_{1\le j\le s}$ be a basis for the free $R$-module $I_r$. For each $j=1,\dots,s$, we let
$$y_j:=\sum_{i=1}^r \Phi_{c_{j,i}}(x_i).$$
Note that each $|y_j|_v<1$ since $(c_{j,1},\cdots,c_{j,r})\in I_r$. Also, each $y_j$ is nonzero since the elements $x_1,\dots,x_r$ form a basis for the free $R$-module $\Gamma$, and not all $c_{j,i}$ are $0$ (for $i=1,\dots,r$).

We know that for $n$ sufficiently large, $(a_{n,1},\cdots,a_{n,r})\in I_r$, and thus each $\sum_{i=1}^r \Phi_{a_{n,i}}(x_i)$ is in the $\Phi(R)$-module generated by $y_1,\cdots,y_s$. So, if $s<r$, we are done by the induction hypothesis. Therefore, assume from now on that $s=r$. So, for each $n\in\N$ such that $|\sum_{i=1}^r \Phi_{a_{n,i}}(x_i)|_v<1$, there exist $b_{n,1},\cdots,b_{n,r}\in R$ such that
$$a_{n,i}=\sum_{j=1}^r b_{n,j} c_{j,i}\text{ for all }i=1,\dots,r.$$
Therefore $\sum_{i=1}^r \Phi_{a_{n,i}}(x_i)=\sum_{j=1}^r \Phi_{b_{n,j}}(y_j)$. Since each $|y_j|_v<1$, we get that for each $b\in R$ we have
\begin{equation}
\label{important equality from equidrin}
\Phi_b(y_j)=by_j+L(b,y_j)\text{ where }|L(b,y_j)|_v<|y_j|_v,
\end{equation}
since each coefficient of $\Phi_b$ is a $v$-adic integer.
Without loss of generality assume that
$$|y_1|_v=|y_2|_v=\cdots = |y_{\ell}|_v=|u|_v^m>|y_{\ell+1}|_v\ge \cdots |y_r|_v\text{ for some $1\le \ell\le r$,}$$
where $u$ is some fixed uniformizer of the maximal ideal of $\OO_v$, and $m\in\N$.
Using \eqref{important equality from equidrin} we conclude that
$$\sum_{j=1}^r \Phi_{b_{n,j}}(y_j)=\sum_{j=1}^{\ell}b_{n,j}y_j+L(n,\overline{y}),$$
where $|L(n,\overline{y})|_v<|u|_v^m$.
For each $j=1,\dots,\ell$, we let $d_j$ be the first nonzero coefficient when expanding $y_j$ with respect to $u$ inside $K_v$. So,
$$y_j=d_j\cdot u^m + \text{ larger powers of $u$},$$
with $|d_j|_v=1$ for each $j=1,\dots,\ell$. Thus
$$\sum_{j=1}^r \Phi_{b_{n,j}}(y_j)=\left(\sum_{j=1}^{\ell} b_{n,j}d_j\right)\cdot u^m + L_1(n,\overline{y}),$$
where $|L_1(n,\overline{y})|_v<|u|_v^m$. Now, since
$$\sum_{i=1}^r \Phi_{a_{n,i}}(x_i)=\sum_{j=1}^r \Phi_{b_{n,j}}(y_j)\to 0\text{ as $n\to \infty$},$$
we conclude that for all $n\in\N$ sufficiently large we have that $|\sum_{j=1}^{\ell} b_{n,j}d_j|_v<1$. For each $j=1,\dots,\ell$, we let $e_j$ be the reduction of $d_j$ modulo $v$, and since $\F_p(t)\subseteq\F$ embeds naturally into the residue field at the place $v$, we conclude that
\begin{equation}
\label{linear dependence relation}
\sum_{j=1}^{\ell}e_jb_{n,j}=0,
\end{equation}
for $n$ sufficiently large. Since each $d_j$ is a $v$-adic unit, we conclude that each $e_j$ is nonzero. The set $S$ of tuples $(b_1,\dots,b_r)$ for which $\sum_{j=1}^{\ell} e_j b_j=0$ is a submodule of $R^r$ of rank $r-1$. Let $\{(f_{j,1},\dots,f_{j,r})\}_{1\le j\le r-1}$ be generators of $S$, and let
$\Gamma_S\subseteq K_{\Phi}$ be the $\Phi(R)$-module spanned by
$$\sum_{i=1}^r \Phi_{f_{j,i}}(y_i)\text{ for each $j=1,\dots,r-1$.}$$
By the induction hypothesis on the rank of the $R$-module $\Gamma_S$ (which is less than $r$), we know that $\Gamma_S$ is discrete $v$-adically. Since for $n$ sufficiently large we have that all $\sum_{i=1}^r \Phi_{a_{n,i}}(x_i)=\sum_{j=1}^r \Phi_{b_{n,j}}(y_j)$ are in $\Gamma_S$, we are done.
\end{proof}

\begin{proof}[Proof of Theorem~\ref{adelic result generic characteristic}.]
Let $\Gamma_0$ be the $\Phi(R)$-submodule of $K$ spanned by the projections of $\Gamma$ on each coordinate of $K_{\Phi}^g$. Since $\Gamma$ is finitely generated, then also $\Gamma_0$ is finitely generated.
We let $v$ be a place in $\Omega$ at which all coefficients of $\Phi_t$ and also each of the finitely many generators  of $\Gamma_0$ are integral. In particular, this means that \emph{all} elements of $\Gamma$ are integral at the place $v$, and also that each coefficient of each $\Phi_a$ is a $v$-adic integer. Note that all but finitely many places $v\in\Omega$ satisfy these conditions.
Proposition~\ref{discrete subgroups} yields that $\Gamma_0$ (and thus $\Gamma_0^g$) is discrete $v$-adically. Therefore $\Gammabar=\Gamma$ and so, $V(K)\cap\Gamma=V(\A)\cap\Gammabar$, as desired.
\end{proof}

%#######################################################################
%#######################################################################

\section{Proof of Theorem~\ref{adelic result special characteristic} when $\dim(V)=0$}
\label{first part}

In this Section, we work under the hypotheses from Theorem~\ref{adelic result special characteristic}. So, $\Phi$ is a Drinfeld module of special characteristic such that $\Phi[t^{\infty}](K^{\sep})$ is finite. In addition, we assume $\Phi$ has positive modular transcendence degree.
First we will prove several preliminary results which describe the topology induced on each finitely generated $\Phi(R)$-module $\Gamma$ by the ad\`{e}lic topology.

We let $\Omega_0$ be the set of all places $v\in\Omega$ at which all coefficients of $\Phi_t$ are integral, while the first and the last nonzero coefficients of $\Phi_t$ are $v$-adic units. Clearly, $\Omega\setminus\Omega_0$ is finite. Note that by our assumption, we obtain that for \emph{all} nonzero $a\in R$, each coefficient of $\Phi_a$ is a $v$-adic integer, while the first and the last coefficients of $\Phi_a$ are $v$-adic units.

\begin{prop}
\label{the quotient is Hausdorff}
For each $v\in\Omega_0$ and for each nonzero $a\in R$, the quotient $K_v^g/\Phi_a \left(K_v^g\right)$ is Hausdorff.
\end{prop}

\begin{proof}
Let $v\in\Omega_0$, and let $a\in R$ be nonzero.
It suffices to show that $K_v/\Phi_a(K_v)$ is Hausdorff, which is equivalent with showing that $\Phi_a(K_v)$ is closed in $K_v$. For this, we need to show that if $\Phi_a(x_n)\to y$ as $n\to \infty$ for some $x_n\in K_v$ and $y\in K_v$, then actually $y\in \Phi_a(K_v)$.

Let $z_n:=x_{n+1}-x_n$. We have that $\Phi_a(z_n)\to 0$ as $n\to\infty$. We claim that the distance between $z_n$ and the nearest torsion point from $\Phi[a]$ tends to $0$ with respect to the $v$-adic topology.

If $t\nmid a$, this is an immediate consequence of Hensel's Lemma. Indeed, reducing modulo $v$, we obtain that the reduction $\overline{z_n}$ modulo $v$ of $z_n$ is a root of the reduced polynomial $\overline{\Phi_a}$ modulo $v$. Now, since $\Phi_a'(x)=a_0\ne 0$ (where $a_0\in\F_p$ such that $t\mid a-a_0$), and $|a_0|_v=1$, we conclude using Hensel's Lemma that there exists a torsion point $u_n\in\Phi[a](K_v)$ such that $|z_n-u_n|_v<1$. Note also that each torsion point is integral at each place $v$ in $\Omega_1$ due to our assumptions regarding the $v$-adic absolute values of the coefficients of $\Phi_t$.
Then
$$|z_n-u_n|_v=|a_0(z_n-u_n)|_v=|\Phi_a(z_n-u_n)|_v=|\Phi_a(z_n)|_v\to 0,\text{ as desired.}$$
Now, assume $t\mid a$. Let $a=c_m t^m+\cdots +c_n t^n$ (with $1\le m\le n$), where each $c_i\in\F_p$, and $c_m,c_n\ne  0$. Letting $\Phi_t(x)=\sum_{j=\ell}^D \gamma_jx^{p^j}$, where $\gamma_{\ell},\gamma_D\ne 0$, then
$$\Phi_a(x)=c_m\gamma_{\ell}^{(p^{m\ell}-1)/(p^{\ell}-1)}x^{p^{m\ell}}+ \sum_{j=m\ell+1}^{nD}\delta_j x^{p^j}. $$
We let $f:=\Phi_a^{1/p^{m\ell}}\in K^{1/p^{m\ell}}[x]$. It is immediate to see that $f(z_n)\to 0$ in the $v$-adic topoogy (by abuse of language, we call also $v$ the induced valuation on $K_v^{1/p^{m\ell}}$ extending $v$). Also, $f(x)$ is a separable additive polynomial in $K^{1/p^{m\ell}}[x]$ whose derivative equals $\delta_0:=c_m\gamma_{\ell}^{(p^{m\ell}-1)/(p^{(m+1)\ell}-p^{m\ell})}$. By our definition of the set $\Omega_0$, we know that $\delta_0$ is a $v$-adic unit. Therefore, by Hensel's Lemma, just as before, we conclude that the distance from $z_n$ to the nearest root $u_n\in K_v^{1/p^{m\ell}}$ of $f(x)=0$ (which is also a torsion point in $\Phi[a]$) converges to $0$ in the $v$-adic topology.

So, for each $n\in\N$, there exists a suitable torsion point $z'_n\in \Phi[a]\left(K_v^{1/p^{m\ell}}\right)$ such that if $y_n=x_n+z'_n$, then the sequence $\{y_n\}_{n\in\N}$ is convergent in $K_v^{1/p^{m\ell}}$. Because $\Phi[a]$ is a finite set, it means that we can extract an infinite sequence $\{n_i\}_{i\in\N}$ of positive integers such that $z'_{n_i}=z'_{n_{i+1}}$ for all $i\in\N$. Hence the subsequence $\{x_{n_i}\}_{i\in\N}$ is convergent in $K_v$; since $K_v$ is a complete space, then $\{x_{n_i}\}_{i\in\N}$ converges to a point $x$ such that $\Phi_a(x)=y$, as desired.
\end{proof}

Let $\Gamma\subseteq \bG_a^g(K)$ be a finitely generated $\Phi(R)$-submodule; we also assume that $\Gamma$ is full in $K$. Let $\Omega_1\subseteq \Omega_0$ be the set of all places $v$ at which the finitely many generators of $\Gamma\subseteq \bG_a^g(K)$ are integral. In particular, this means that all points of $\Gamma$ are $v$-adic integral for $v\in\Omega_1$; also note that $\Omega\setminus\Omega_1$ is finite.
Without loss of generality (see also Proposition~\ref{essential lemma}) we may assume $\A_K$ is the restricted product $\prod_{v\in\Omega_1} (K_v,\OO_v)$.

\begin{prop}
\label{adelic topology is given by powers of t}
Let $\Gamma$ be a full, finitely generated $\Phi(R)$-submodule of $K^g$, and let $v\in\Omega_1$. Then the $v$-adic topology on $\Gamma$ is at least as strong as the topology induced by all subgroups $\Phi_{t^n}(\Gamma)$ for $n\in\N$.
\end{prop}

\begin{proof}
All we need to show is that for each $n\in\N$, there exists an open subset $U_n$ of $K_v^g$ such that $U_n\cap \Gamma \subseteq \Phi_{t^n}(\Gamma)$.

Let $n\in\N$. We know that $\Phi[t^{\infty}](K^{\sep})$ is finite, which means that there exists $m\in\N$ such that $\Phi[t^{\infty}](K^{\sep})\subseteq \Phi[t^m]$. Let
$$\pi: \Gamma \lra K_v^g/\Phi_{t^{n+m}}(K_v^g)$$
be the composition of the natural injection of $\Gamma$ into $K_v^g$ followed by the canonical projection map on the quotient $K_v^g/\Phi_{t^{n+m}}(K_v^g)$; clearly, $\pi$ is a continuous map. According to Proposition~\ref{the quotient is Hausdorff}, the quotient $K_v^g/\Phi_{t^{n+m}}(K_v^g)$ has an induced Hausdorff topology, which means that there exists an open subset $U_n$ of $K_v^g$ such that $U_n\cap \Gamma=\ker(\pi)$ (we also use the fact that $\Gamma/\Phi_{t^{n+m}}(\Gamma)$ is finite, and hence the induced topology on $\Gamma/\Phi_{t^{n+m}}(\Gamma)$ as a subset of $K_v^g/\Phi_{t^{n+m}}(K_v^g)$ is the discrete topology).

Now, for any $y\in U_n\cap \Gamma$, there exists $x\in K_v^g\cap \Kbar^g$ such that $y=\Phi_{t^{n+m}}(x)$. However, $\Kbar\cap K_v\subseteq K^{\sep}$ (see \cite[Lemma 3.1]{Poonen-Voloch}); so, $x\in (K^{\sep})^g$. Now, for any automorphism $\sigma\in \Gal(K^{\sep}/K)$, we have that $\Phi_{t^{n+m}}(x^{\sigma}-x)=0$ (since $y=\Phi_{t^{n+m}}(x)\in K^g$), which means that $x^{\sigma}-x\in (\Phi[t^{n+m}](K^{\sep}))^g$. So, in particular $x^{\sigma}-x\in (\Phi[t^m](K^{\sep}))^g$, which means that $\Phi_{t^m}(x)\in K^g$. However, since $\Phi_{t^m}(x)\in K^g$ and $\Phi_{t^n}(\Phi_{t^m}(x))=y\in\Gamma$, while $\Gamma$ is a full $\Phi(R)$-submodule in $K$, we conclude that also $\Phi_{t^m}(x)\in\Gamma$, and thus $y=\Phi_{t^n}(\Phi_{t^m}(x))\in \Phi_{t^n}(\Gamma)$. Therefore
$U_n\cap\Gamma\subseteq \Phi_{t^n}(\Gamma)$, as desired.
\end{proof}

\begin{prop}
\label{the closure is the same}
Let $\Gamma\subseteq K^g$ be a finitely generated $\Phi(R)$-submodule which is full in $K$. For each $v\in\Omega_1$ we let $\Gamma_v$ be the closure of $\Gamma$ in $K_v^g$. Then $\Gamma_v\cap K^g=\Gamma$.
\end{prop}

\begin{proof}
One inclusion is obvious. Now, let $y\in\Gamma_v\cap K^g$. We let $\Gamma_1$ be the full divisible hull in $K$ of the finitely generated $\Phi(R)$-module $\Gamma_0$ spanned by $\Gamma$ and $y$. There exists an infinite sequence of distinct elements $x_n\in\Gamma$ such that $x_n\to y$ in $K_v^g$. So, by Proposition~\ref{adelic topology is given by powers of t}, we have that for all $j\in\N$, there exists $n_j\in\N$ such that
\begin{equation}
\label{for each j}
x_{n_j}-y=\Phi_{t^j}(y_j)\text{ for some $y_j\in\Gamma_1$.}
\end{equation}
Since $\Gamma_1$ is finitely generated as a $\Phi(R)$-module, and $\Gamma_1$ is the full divisible hull in $K$ of $\Gamma_0$, we conclude that there exists a nonzero $b\in R$ such that $\Phi_b(\Gamma_1)\subseteq \Gamma_0$. We assume $b=t^{j_0}\cdot b_0$, for some $j_0\in\N$, and some $b_0\in R$ which is not divisible by $t$. Now, we use \eqref{for each j} for $j=j_0+1$, and obtain
\begin{equation}
\label{j=j_0+1}
\Phi_{b_0}(x_{n_j})-\Phi_{b_0}(y)=\Phi_{b_0}\left(\Phi_{t^{j_0+1}}(y_j)\right).
\end{equation}
Because $\Phi_{b_0t^{j_0}}(y_j)\in\Gamma_0$, we conclude that there exist $z_j\in\Gamma$ and $c_j\in R$ such that $\Phi_{b_0t^{j_0}}(y_j)=z_j+\Phi_{c_j}(y)$. Therefore, \eqref{j=j_0+1} reads
$$\Phi_{b_0}(x_{n_j})-\Phi_{b_0}(y)=\Phi_t(z_j)+\Phi_{tc_j}(y).$$
Hence, $\Phi_{tc_j+b_0}(y)=\Phi_{b_0}(x_{n_j})-\Phi_t(z_j)\in\Gamma$. Since $t\nmid b_0$, we get that $tc_j+b_0\ne 0$, which yields that $y\in\Gamma$ because $\Gamma$ is full in $K$.
\end{proof}

\begin{prop}
\label{topology is given by prime to t}
The topology induced on $\Gamma$ by the ad\`{e}lic topology is at least as strong as the one given by all subgroups $\Phi_a(\Gamma)$ for $a\in R$ relative prime to $t$.
\end{prop}

\begin{proof}
Let $a\in R$ relative prime to $t$. We have to show that there exists an open subset $U_a$ in the ad\`{e}lic topology such that $U_a\cap\Gamma \subseteq \Phi_a(\Gamma)$. Since $\Phi_a$ is separable (more precisely, $\Phi_a'$ is a constant $v$-adic unit for all $v\in\Omega$), using Hensel's Lemma, we obtain that $\Phi_a(\OO_v^g)$ is an open subgroup of $K_v^g$ for each $v\in\Omega_1$. Therefore it suffices to show that there exist finitely many places $v_i\in\Omega_1$ (with $1\le i\le n_a$ for some $n_a\in\N$) such that for all $y\in\Gamma$ satisfying $y\in\Phi_a(\OO_{v_i}^g)$ for $1\le i\le n_a$, we have $y\in \Phi_a(\Gamma)$. Since $\Gamma$ is full in $K$, we only have to show that $y\in \Phi_a(K^g)$.

Now, assume there is no such finite set of places $v_i$ as above. This means that for any finite subset $S$ of places contained in $\Omega_1$, we may find $y=y_S\in \Gamma\setminus \Phi_a(\Gamma)$ such that $y\in \Phi_a(\OO_v^g)$ for each $v\in S$.  Note that the quotient $\Gamma/\Phi_a(\Gamma)$ is finite (since $\Gamma$ is a finitely generated $\Phi(R)$-module and $a\ne 0$); let $y_1,\dots,y_s\in \Gamma$ be a set of representatives for the nonzero cosets of $\Gamma/\Phi_a(\Gamma)$.

For each coset $y_i + \Phi_a(\Gamma)$ for $i=1,\dots,s$, we let $S_{y_i}$ be the set of places $v\in\Omega_1$ such that $y_i\notin \Phi_a(\OO_v^g)$; note that our definition does not depend on the particular choice of the class representative for the particular coset of $\Phi_a(\Gamma)$ in $\Gamma$.  If each $S_{y_i}$ is nonempty, then for each $i=1,\dots,s$, choose $v_i\in S_{y_i}$ and let $S:=\{v_1,\dots,v_s\}$. By our assumption, if $y\in \Phi_a(\OO_{v_i}^g)$ for each $i=1,\dots,s$, then $y\in \Phi_a(\Gamma)$, as desired.

So, from now on assume there exists a coset $y+\Phi_a(\Gamma)$ of $\Gamma$ (as above) such that $S_y$ is empty, i.e. for \emph{all}  places $v\in\Omega_1$, we have $y\in \Phi_a(\OO_v^g)$. So, even though $y\notin \Phi_a(K^g)$ (note that $\Gamma$ is full in $K$), for all but finitely many places $v\in\Omega$ we still have that $y\in \Phi_a(\OO_v^g)$. We will show next that this is impossible.
At the expense of passing to the coordinates in $\bG_a^g$ of each point, we may assume $g=1$.
 In particular this means that the polynomial $\Phi_a(X)-y\in K[X]$ has no linear irreducible factor over $K$. Since $\F$ is a finite extension of $\F_p(t)$, then $\F$ is a Hilbertian field (see \cite[Ch. 9]{Lang}). This yields that there exist infinitely many places $v\in\Omega_1$ such that the reduction modulo $v$ of the polynomial $\Phi_a(X)-y$ has no irreducible linear factor over $\F_v$. This contradicts our assumption that for all but finitely many places $v\in\Omega_1$ we have $y\in \Phi_a(\OO_v)$ since that would mean that for all but finitely many places $v\in\Omega_1$, the reduction modulo $v$ of $y$ belongs to $\Phi_a(\F_v)$ (since $\F_v$ is the residue field of $K$ and also of $\OO_v$ modulo $v$). This concludes the proof of our Proposition~\ref{topology is given by prime to t}.
\end{proof}

Combining Propositions~\ref{adelic topology is given by powers of t} and \ref{topology is given by prime to t} we obtain the following result.
\begin{prop}
\label{adelic topology is given by all a}
The induced ad\`{e}lic topology on $\Gamma$ is at least as strong as the topology given by all subgroups $\Phi_a(\Gamma)$ for all $a\in R$.
\end{prop}

\begin{prop}
\label{adelic torsion is regular torsion}
Let $\Gamma\subseteq K_{\Phi}^g$ be any finitely generated $\Phi(R)$-submodule which is full in $K$. Then $\Gammabar_{\tor}=\Gamma_{\tor}$, i.e. all torsion points contained in $\Gammabar$ are actually contained in $\Gamma$.
\end{prop}

\begin{proof}
The proof is similar to the proof of \cite[Lemma 3.7]{Poonen-Voloch}. Let $\Gamma_0$ be the subgroup of all $y\in\Gamma$ such that for all $v\in\Omega_1$ the reduction of $y$ modulo $v$ is in $\left(\bar{\Phi}_{\tor}(\F_v)\right)^g$, where $\bar{\Phi}$ is the reduction of $\Phi$ modulo $v$. It is immediate to see that $\Gamma_0$  is full in $K$, and thus it contains $\Gamma_{\tor}$.

Now, since $\Gamma$ is a finitely generated $\Phi(R)$-submodule, there exist finitely many places $v_i\in\Omega_1$ for $i=1,\dots,\ell$ such that $\Gamma_0$ is the kernel of the natural reduction map
$$\Gamma\lra \prod_{i=1}^{\ell} \left(\F_{v_i}/\bar{\Phi}_{\tor}(\F_{v_i})\right)^g.$$
Therefore $\Gamma/\Gamma_0$ embeds into the free $\Phi(R)$-submodule
$$\prod_{i=1}^{\ell} \left(\F_{v_i}/\bar{\Phi}_{\tor}(\F_{v_i})\right)^g.$$
We conclude that there exists a free submodule $\Gamma_1\subseteq \Gamma$ such that $\Gamma=\Gamma_0\oplus\Gamma_1$. We claim that
\begin{equation}
\label{structure adelic}
\Gammabar=\bar{\Gamma_0}\oplus \Gamma_1
\end{equation}
Indeed, for each sequence $\{x_n+y_n\}_{n\in\N}$ where $x_n\in\Gamma_0$ and $y_n\in\Gamma_1$ which converges in $\A_K^g$, we obtain that for $m$ and $n$ sufficiently large we have that $y_n-y_m$ and $x_m-x_n$ have the same reduction modulo $v_i$ for each $i=1,\dots,\ell$. Therefore, $y_n-y_m\in\Gamma_0$, and thus $y_n=y_m$ since $\Gamma=\Gamma_0\oplus \Gamma_1$.

So $\Gammabar=\bar{\Gamma_0}\oplus \Gamma_1$ and since $\Gamma_1$ is a free $\Phi(R)$-module, we obtain that $\Gammabar_{\tor}=\left(\bar{\Gamma_0}\right)_{\tor}$. We claim that the induced ad\`{e}lic topology on $\Gamma_0$ is the one given by all subgroups $\Phi_a(\Gamma_0)$ for all $a\in R$. By Proposition~\ref{adelic topology is given by all a} we know that the ad\`{e}lic topology induces on $\Gamma_0$ a topology at least as strong as the one given by all subgroups $\Phi_a(\Gamma_0)$. Now we check the converse statement. Indeed, for each  $y\in\Gamma_0$ and for each $v\in\Omega_1$, there exists a nonzero $a_v\in R$ such that $\Phi_{a_v}(y)$ reduces to $0$ modulo $v$. Then $\Phi_{t^na_v}(y)$ converges $v$-adically to $0$  as $n\to\infty$, because $v$ is a place of integrality for all coefficients of $\Phi_{t^na_v}$ and moreover $\Phi_t$ has no linear or constant term. Moreover, let $\{y_i\}_{1\le i\le s}$ be a finite set of generators for $\Gamma_0$ as a $\Phi(R)$-module, and let $a_{i,v}\in R$ be as above for each corresponding $y_i$ and each place $v\in \Omega_1$. Letting $b_v:=\prod_{i=1}^{s} a_{i,v}$, we obtain that
$$\Phi_{t^nb_v}(y)\text{ coverges $v$-adically to $0$, as $n\to\infty$, for all $y\in\Gamma_0$.}$$
This proves our claim about the induced ad\`{e}lic topology on $\Gamma_0$. Thus
\begin{equation}
\label{important bar}
\bar{\Gamma_0}\isomto \Gamma_0\tensor_R \hat{R},
\end{equation}
where $\hat{R}$ is the restricted direct product of all completions of $R$ at its maximal ideals. Since $R/aR\isomto \hat{R}/a\hat{R}$ for each nonzero $a\in R$, we conclude that
$$\Gammabar_{\tor}=\left(\bar{\Gamma_0}\right)_{\tor}\isomto (\Gamma_0)_{\tor}=\Gamma_{\tor},$$
which concludes the proof of Proposition~\ref{adelic torsion is regular torsion}.
\end{proof}

The following result is an immediate consequence of our proof of Proposition~\ref{adelic torsion is regular torsion}.
\begin{cor}
\label{the quotient is Hausdorff bis}
For each nonzero $a\in R$, we have $\Gamma/\Phi_a(\Gamma)\isomto \Gammabar/\Phi_a(\Gammabar)$.
\end{cor}

\begin{proof}
The claimed result follows immediately from \eqref{structure adelic} and \eqref{important bar} using that $\hat{R}/a\hat{R}\isomto R/aR$.
\end{proof}

Now we are ready to prove Theorem~\ref{adelic result special characteristic} when $\dim(V)=0$.

\begin{prop}
\label{zero-dimensional case}
Theorem~\ref{adelic result special characteristic} holds if $V$ is a zero-dimensional subvariety of $\bG_a^g$.
\end{prop}

\begin{proof}
Clearly, it suffices to show that $V(\A_K)\cap \Gammabar\subseteq V(K)\cap\Gamma$ since the reverse inclusion is obvious.

First we show that it suffices to prove our result when we replace $K$ by a finite extension $L$, and replace $\Gamma$ by its full divisible hull $\Gamma_L$ inside $L_{\Phi}^g$. Indeed, we denote by $\A_L$ the ad\`{e}les for the function field $L$ (with respect to the places which extend the places contained in $\Omega_1$), and we assume that
$$V(\A_L)\cap \bar{\Gamma_L}\subseteq V(L)\cap\Gamma_L.$$
Then
$$V(\A_K)\cap\Gammabar\subseteq V(\A_L)\cap \bar{\Gamma_L} \subseteq V(L)\cap \Gamma_L\subseteq V(L).$$
Therefore $V(\A_K)\cap\Gammabar\subseteq V(L)\cap V(\A_K)=V(K)$ because $L\cap\A_K=K$ by \cite[Lemma 3.2]{Poonen-Voloch}. So, $V(\A_K)\cap\Gammabar\subseteq K^g\cap \Gammabar=\Gamma$ by Proposition~\ref{the closure is the same}. Therefore
$$V(\A_K)\cap\Gammabar\subseteq V(K)\cap \Gamma\text{, as desired.}$$

Secondly, we show that for any finite extension $L$ of $K$, the hypothesis that $\Phi[t^{\infty}](L^{\sep})$ is finite holds. Indeed, this claim is obvious when $L$ is a separable extension of $K$; so it suffices to prove it when $L$ is a purely inseparable extension of $K$. Hence, assume $L=K^{1/p^r}$ for some positive integer $r$. Since $\Phi_t$ is inseparable, we obtain that $\Phi_{t^r}(L^{\sep})\subseteq K^{\sep}$. Because $\Phi[t^{\infty}](K^{\sep})$ is finite, we know that there exists a positive integer $m$ such that $\Phi[t^{\infty}](K^{\sep})\subseteq \Phi[t^m]$. Therefore, $\Phi[t^{\infty}](L^{\sep})\subseteq \Phi[t^{r+m}]$, as desired.

Now, since we proved that we may replace $K$ by a finite extension (and replace $\Gamma$ by its full divisible hull inside this extension), we may assume that $V(\Kbar)=V(K)$ (recall  our hypothesis is that $\dim(V)=0$).
Let $\{x_n\}_{n\in\N}\subseteq \Gamma$ be a sequence which converges ad\`{e}lically to a point $(y_v)_{v\in\Omega_1}\in V(\A_K)$, where $y_v\in V(K_v)$ for all places $v$. Actually, $y_v\in V(K)$ since $V$ is a zero dimensional variety whose algebraic points are all defined over $K$. So, we know that for all places $v\in\Omega_1$, we have
$$x_n\to y_v\text{ in the $v$-adic topology.}$$
We claim that $y_v-y_w\in\Gamma_{\tor}$ for any given two places $v$ and $w$ from $\Omega_1$. Since $\Gamma$ is full in $K$, and since $y_v,y_w\in K^g$, all we need to show is that $y_v-y_w$ is a torsion point for $\Phi$.
By Proposition~\ref{adelic topology is given by powers of t}, we conclude that for each $j\in \N$, there exists $n_j\in\N$ and  there exist $z_{v,j},z_{w,j}\in K^g$ such that
$$x_{n_j}-y_v=\Phi_{t^j}(z_{v,j})\text{ and }x_{n_j}-y_w=\Phi_{t^j}(z_{w,j}).$$
So, $y_v-y_w=\Phi_{t^j}(z_{w,j}-z_{v,j})\in \Phi_{t^j}(K^g)$. Since we can do this for \emph{all} $j\in\N$, and $K_{\Phi}^g$ is a direct product of a finite torsion submodule with a free submodule of rank $\aleph_0$, we conclude that $y_v-y_w\in\Phi_{\tor}(K)^g=\Gamma_{\tor}$, as claimed.

Because $\Gamma_{\tor}$ is a finite set, and because $y_v-y_w\in \Gamma_{\tor}$ for each pair $(v,w)\in\Omega_1\times \Omega_1$, we conclude that there exists $y\in V(K)$ and there exist infinitely many places $v\in \Omega_1$ such that $y=y_v$. Then, as shown above, the sequence $\{x_n-y\}_{n\in\N}$ converges  to an ad\`{e}lic  torsion point for $\Phi$. By Proposition~\ref{adelic torsion is regular torsion} applied to the full divisible hull in $K$ of the $\Phi(R)$-submodule generated by $\Gamma$ and $y$, we conclude that $\{x_n-y\}_{n\in\N}$ converges ad\`{e}lically to a torsion point $z_0\in\Phi_{\tor}(K)^g=\Gamma_{\tor}$. However, since the reduction of $z_0$ modulo $v$ equals $0$ for infinitely many places $v$, we conclude that $z_0=0$. Therefore, the sequence $\{x_n\}_{n\in\N}$ converges ad\`{e}lically to the point $y\in V(K)$. Using Proposition~\ref{the closure is the same}, we obtain that $y\in\Gamma$, which yields the desired conclusion: $V(\A_K)\cap\Gammabar=V(K)\cap\Gamma$.
\end{proof}

%#######################################################################
%#######################################################################
%#######################################################################

\section{A uniform Mordell-Lang version for Drinfeld modules}
\label{second part}

In this Section we continue with the same notation as in Section~\ref{first part}. Our main goal is to prove the following result, which follows from the unpublished notes of the second author.
\begin{prop}
\label{uniform DML}
Let $K$ be a function field over $\F$ of arbitrary finite transcendence degree, and let $\Phi$ be a Drinfeld module of special characteristic defined over $K$. Assume $\Phi$ has positive modular transcendence degree, and that $\Phi[t^{\infty}](K^{\sep})$ is finite. Let $\Gamma$ be a full,  finitely generated $\Phi(R)$-submodule of $K_{\Phi}^g$.
Let $V\subseteq\bG_a^g$ be an affine subvariety defined over $K$, and assume that $V$ contains no translate of a positive dimensional algebraic subgroup of $\bG_a^g$.

Then there exists a zero-dimensional subvariety $W\subseteq V$ defined over $K$ such that $V(\A_K)\cap\Gammabar\subseteq W(\A_K)$.
\end{prop}

Proposition~\ref{uniform DML} allows us to reduce Theorem~\ref{adelic result special characteristic} to the case $\dim(V)=0$, which was already proven in Proposition~\ref{zero-dimensional case}.

In order to prove Proposition~\ref{uniform DML} we shall require a uniformity result used by Hrushovski in his proof \cite{Hr} of the positive characteristic function field Mordell-Lang conjecture but adapted here for Drinfeld modules. Our results are valid for any inseparable endomorphism $\psi$ of $\bG_a$ defined over $K$; for our application, we have $\psi=\Phi_t$. For each $n\in\N$ we denote by $\psi^n$ the $n$-th iterate of $\psi$ when composed with itself, and we extend the action of $\psi$ on $\bG_a^g$ for any $g\in\N$.

First we let $\mathcal{B}$ be a $p$-basis for $K$, i.e. a basis for $K$ as $K^p$-vector space.  We say that an extension of fields $L/K$ is \emph{fully separable} if $\mathcal{B}$ continues to be a $p$-basis for $L$.  With this definition, a subextension of a fully separable extension need not be fully separable itself.

For $L/K$ a fully separable, separably closed extension field of $K$, define $\psi^\infty (L) := \bigcap_{n \geq 0} \psi^{n}(\Ga(L)) \subseteq \Ga(L)$.

\begin{prop}
\label{automatic uniformity}
Given an algebraic subvariety $X \subseteq \bG_a^g$ having the property that $X$ does not contain the translate of any positive dimensional algebraic subgroup of $\Ga^g$, there are numbers $\ell$ and $m$ such that for any fully separable field extension $L/K$ and point $a \in \bG_a^g(L)$ the cardinality of $(a + X(L)) \cap \psi^m (\Ga^g(L))$ is at most $\ell$.
\end{prop}

The proof of this proposition relies on the main dichotomy theorem on minimal groups type definable in separably closed fields, though in this particular case the Pillay-Ziegler~\cite{PZ} geometric argument using jet spaces may be employed rather than the Zariski geometries argument needed in general.  To apply the Pillay-Ziegler argument following their proof of the positive characteristic Mordell-Lang conjecture for ordinary abelian varieties, we need to know that the generic type of $\psi^\infty$ is \emph{very thin} in their language.  That is, if $L/K$ is a fully separable extension of $K$ and $a \in \psi^\infty(L)$ then the smallest fully separable subfield of $L$ containing $a$ and $K$ is finitely generated as a field over $K$.  The argument of~\cite[Lemma 6.4]{PZ}  in which it is shown that if $A$ is ordinary then $p^\infty A$ is very thin applies to $\psi^\infty$.  That is, we may write $\psi = \widetilde{\psi} \circ F$ where $F$ is a Frobenius map corresponding to some power of $p$ and $\widetilde{\psi}$ is separable.  Following the proof of \cite[Lemma 6.4]{PZ} replacing the Verschiebung by $\widetilde{\psi}$, one sees that $K(a)$ is already fully separable.

With this preamble, let us state precisely the required result and then explain how to deduce Proposition~\ref{uniform DML} from the main theorem.

\begin{thm}
\label{AWZ}
With our hypotheses on $\psi$, the type definable group $\psi^\infty$ is modular.  In particular, for any separably closed, fully separable extension field $L/K$, and any algebraic variety $Y \subseteq \bG_a^g$ defined over $L$, the set $Y(L) \cap \psi^\infty(L)^g$ is a finite union of translates of subgroups (possibly trivial) of $\psi^\infty(L)^g$.
\end{thm}

The proof of Theorem~\ref{AWZ} follows identically as the proof of the main theorem of~\cite{Hr} and is sketched by the second author in the Arizona Winter School notes~\cite{Sc}.    As noted above, in this case one could also give a geometric proof using a variant of the Gauss map.

Given Theorem~\ref{AWZ}, let us show how to deduce Proposition~\ref{automatic uniformity}.   This argument also appears in~\cite{Hr}.
\begin{proof}[Proof of Proposition~\ref{automatic uniformity}.]
If no such bounds $\ell$ and $m$ existed, then the following set of sentences would be consistent:

\begin{itemize}
\item $L$ is a separably closed, fully separable extension of $K$. (This is expressible by a countable set of sentences in the language of rings augmented by constant symbols for the elements of $K$.)
\item $a \in \bG_a^g(L)$
\item $c_i \in \bG_a^g(L)$ (for $i \in {\mathbb N}$)
\item $c_i \neq c_j$ for $i \neq j$
\item $(\exists y \in \bG_a^g(L)) \psi^m (y) = c_i$ for each $i \in {\mathbb N}$ and $m \in {\mathbb N}$
\item $c_i \in a + X(L)$
\end{itemize}

By the compactness theorem, we obtain a fully separable, separably closed field extension $L/K$ and a point $a \in \bG_a^g(L)$ so that $(a + X(L)) \cap \psi^\infty(L)^g$ is infinite. By the modularity of $\psi^\infty(L)$ it follows that there is an infinite group $H \leq \psi^{\infty}(L)^g$ and a point $b \in \psi^\infty(L)^g$ such that $b + H \subseteq a + X(L)$.  But then $X$ contains the translate of the Zariski closure of $H$ by $(b -a)$ contrary to our hypothesis on $X$.
\end{proof}

Now we are ready to prove Proposition~\ref{uniform DML}.
\begin{proof}[Proof of Proposition~\ref{uniform DML}.]
Let $m$ and $\ell$ be the numbers given by Proposition~\ref{automatic uniformity} for $X=V$ and $\psi=\Phi_t$.  By Corollary~\ref{the quotient is Hausdorff bis}, the group $\overline{\Gamma}/\Phi_{t^m} \overline{\Gamma}$ is finite and each coset of $\Phi_{t^m} \overline{\Gamma}$ contains an element of $\Gamma$. Let $\gamma_1, \ldots, \gamma_s \in \Gamma$ so that $\overline{\Gamma} = \bigcup_{i=1}^s \gamma_i + \Phi_{t^m} \overline{\Gamma}$.    By Proposition~\ref{automatic uniformity}, for each $i$ there is a zero dimensional variety $W_i$ defined over $K$ for which $V(L) \cap (\gamma_i + \Phi_{t^m} \bG_a^g(L)) \subseteq W_i(L)$ for any fully separable extension field $L/K$; clearly we may assume $W_i\subset V$.  Let
$$W := \bigcup_{i=1}^s W_i.$$
Clearly each $K_v$ is a fully separable extension of $K$. So, if $b=(b_v)_{v\in\Omega_1} \in V(\A_K) \cap \overline{\Gamma}$, we have $b_v \in V(K_v) \cap (\bigcup_{i=1}^s \gamma_i + (\Phi_{t^m} \bG_a^g(K_v))) \subseteq W(K_v)$.  Hence, $b \in W({\mathbf A}_K)$ as desired.
\end{proof}

Now we are ready to prove Theorem~\ref{adelic result special characteristic} for any subvariety $V$ which contains no translate of a positive dimensional algebraic subgroup of $\bG_a^g$.

\begin{proof}[Proof of Theorem~\ref{adelic result special characteristic}.]
By Proposition~\ref{uniform DML}, there exists a zero-dimensional affine subvariety $W$ of $V$ also defined over $K$ such that $V(\A_K)\cap\Gammabar\subseteq W(\A_K)$. On the other hand, by Proposition~\ref{zero-dimensional case}, we have that $W(\A_K)\cap\Gammabar=W(K)\cap\Gamma$. So, $V(\A_K)\cap\Gammabar\subseteq W(\A_K)\cap\Gammabar=W(K)\cap\Gamma$.
On the other hand, since $W\subseteq V$, we conclude that $V(\A_K)\cap\Gammabar\subseteq V(K)\cap\Gamma$. But always $V(K)\cap\Gamma\subseteq V(\A_K)\cap\Gammabar$, which concludes the proof of Theorem~\ref{adelic result special characteristic}.
\end{proof}

%#######################################################################

%\bibliographystyle{plain}

\end{document}